\documentclass[11pt,leqno,amscd,amssymb,verbatim, url]{amsart}
\usepackage{amsfonts,amssymb}
\usepackage{amsmath,amscd}
\newtheorem{thm}{Theorem}[section]
\usepackage{amsfonts,latexsym}
\newtheorem{lem}[thm]{Lemma}
\newtheorem{cor}[thm]{Corollary}

\newtheorem{prop}[thm]{Proposition}
\theoremstyle{definition}

\newtheorem{rem}[thm]{Remark}

\numberwithin{equation}{thm}
\usepackage{wrapfig}

\newcommand{\head}{{\text{\rm hd}}\,}

\newcommand{\fa}{{\mathfrak a}}

\newcommand{\Ext}{{\text{\rm Ext}}}

\newcommand{\St}{\text{\rm St}}

\newcommand{\gr}{\operatorname{ {gr}}}

\newcommand{\id}{\operatorname{Id}}

\newcommand{\End}{\operatorname{End}}

\newcommand{\ind}{\operatorname{ind}}

\newcommand{\res}{\operatorname{res}}

\newcommand{\soc}{\operatorname{soc}}

\newcommand{\rad}{\operatorname{rad}}
\newcommand{\Aut}{\operatorname{Aut}}

\newcommand{\blist}{\begin{list}{\rom{(\roman{enumi})}}{\setlength
{\leftmarg in}{0em} \setlength{\itemindent}{7ex}
\setlength{\labelsep}{2ex}\setlength{\listparindent}{\parindent}
\usecounter{enumi}}}
\newcommand{\elist}{\end{list}}

\begin{document}
\title[Variations on a theme of Cline and Donkin ]{\large {\bf Variations on a theme of Cline and Donkin}}

\author{Brian J. Parshall}
\address{Department of Mathematics \\
University of Virginia\\
Charlottesville, VA 22903} \email{bjp8w@virginia.edu {\text{\rm
(Parshall)}}}
\author{Leonard L. Scott}
\address{Department of Mathematics \\
University of Virginia\\
Charlottesville, VA 22903} \email{lls2l@virginia.edu {\text{\rm
(Scott)}}}
\thanks{Research supported in part by the National Science
Foundation} \subjclass{Primary 20G05}
\begin{abstract} Let $N$ be a normal subgroup of a group $G$. An $N$-module $Q$ is called $G$-stable provided that $Q$ is equivalent to the twist $Q^g$ of $Q$ by $g$, for every $g\in G$. If the action of $N$ on $Q$ extends to an action of $G$ on $Q$, then $Q$ is obviously $G$-stable, but the converse need not hold. A famous conjecture in the modular representation theory of reductive algebraic groups $G$ asserts that the (obviously $G$-stable) projective indecomposable modules (PIMs) $Q$ for the Frobenius kernels $G_r$ ($r\geq 1$) of $G$ have a $G$-module structure. It is sometimes just as useful (for a general module $Q$) to know that a finite direct sum $Q^{\oplus n}$ of $Q$ has a compatible $G$-module structure. In this paper, this property is called numerical stability. In recent work \cite{PS}, the authors established numerical stability in the special case of PIMs. We provide in this paper a more general context for that result, working in the context of $k$-group schemes and a suitable version of $G$-stability, called strong $G$-stability.
 Among our results here is the determination of necessary and sufficient conditions for the existence of a compatible $G$-module structure on a strongly $G$-stable $N$-module, in the form of a cohomological obstruction which must be trivial precisely when the $G$-module structure exists. Our main result is achieved by giving an approach to killing the obstruction by tensoring with certain
 finite dimensional $G/N$-modules.

   \end{abstract}
\maketitle
\section{Introduction}
Let $G$ be a reductive algebraic group over an algebraically closed field $k$ of positive
characteristic $p>0$. Assume $G$ is split over the prime field ${\mathbb F}_p$. For $r\geq1$, let $G_r$ be the  $r$th Frobenius kernel of $G$.
A still open conjecture, published in 1973 and due to Humphreys and Verma, asserts that PIMs  $Q$ for $G_r$ have compatible $G$-module structures; see \cite[p. 325]{Jan}. There is  a sharper conjecture, due to Donkin \cite{Donkin2},
that posits that these PIMs all arise as the restrictions to $G_r$ of specific tilting modules for $G$ (and hence have
a rational $G$-module structure). For $p\geq 2h-2$ ($h$ the Coxeter number of $G$) this
conjecture is valid. In a recent paper \cite{PS}, the authors proved a ``stable" version of Donkin's conjecture
valid in all characteristics.  More precisely, we proved that there is a positive integer $n$ such that the direct
sum $Q^{\oplus n}$ of $n$ copies of $Q$ is a $G$-module, and can be taken to be a tilting module for $G$.\footnote{In fact,
 \cite{PS} shows that if $\nu$ is an $r$-restricted dominant weight, then Donkin's conjecture holds for the $G_{n+r}$-PIM $Q(\nu+p^r(p^n-1)\rho)$, for all $n\gg 0$.
 } This result played an important role in our work there in finding bounds
for $\Ext$-groups.  The question now arises, ignoring issues of tilting modules,
whether some general version of our stability theorem just mentioned for PIMs might be set in a broader theoretical context, and
might be valid for a wider class of $G_r$-modules. In this paper, we prove the following result.

\begin{thm}\label{MainTheorem}  Let $Q$ be a finite dimensional rational $G_r$-module such that
 $Q$ is a $G_r$-direct summand
of a rational $G$-module $M$ in such a way that the $G_r$-socle $\soc^{G_r}Q$ of $Q$ is a $G$-submodule of $M$.
Then, for some positive integer $n$, the $G_r$-module structure on
$Q^{\oplus n}$ extends to a rational $G$-module structure.  In addition, it can be assumed that $\soc^{G_r}Q$
is a $G$-submodule of one of the summands $Q$. \end{thm}

When $Q$ is the injective hull\footnote{All projective $G_r$-modules are injective and vice versa.} of an irreducible $G_r$-module $L$, the hypothesis is easily verified, simply
by embedding $Q$ in the (infinite dimensional) $G$-injective hull of its socle, and taking
$M$ to be the finite dimensional $G$-module generated by the image of $Q$. (As is well-known, this can also be done more concretely
by embedding $Q$ into a rational $G$-module of the form $L'\otimes \St_r$, with $\St_r$  the $r$th Steinberg module and $L'$ a suitably chosen rational $G$-module,  but the present argument seems more theoretically satisfactory.) Hence, as a conclusion of the theorem,
we obtain a second proof to a key result in  \cite{PS}, guaranteeing the existence of certain rational $G$-modules in arbitrary
characteristic which behave much like PIMs for $G_r$.\footnote{See \cite[Cor. A.5]{PS}. The $G$-module constructed there has similar properties
                to the module $M$ in Theorem \ref{MainTheorem}. It is not the tilting module mentioned above, but is
                constructed from it, by tensoring with a twisted Steinberg module.  See the proof of Lemma 
                \ref{help} below for a similar strategy.}  We hope the broader context provided by
                the proof in this paper may have
further use. One consequence already is a cohomological obstruction theory for the original problem of extending $G_r$-modules to $G$-modules, applicable
not only to PIMs but to many of their natural submodules (e.~g., those in the socle series); see Corollary \ref{cortoendofsec3}.

Now let $Q$ be any finite dimensional rational $G_r$-module. For $g\in G$, let $Q^g$ be the rational $G_r$-module obtained by twisting the action of $G_r$ on $Q$ through
conjugation by $g$ (the left adjoint action of $g$ on $G_r$). The module $Q$ is called $G$-stable if $Q\cong Q^g$, for all $g\in G$. If
$Q$ satisfies the conclusion of the theorem, then a Krull-Schmidt argument shows that $Q$ is $G$-stable. As we will see
in Lemma \ref{biglemma}, which is inspired by work of Donkin \cite{Donkin1}, the hypothesis of the theorem actually
implies a strong version of $G$-stability, and is even equivalent to it; see \S4.3. It is this stronger notion that we consider further, obtaining in Theorem \ref{realmaintheorem}
a kind of necessary and sufficient condition for the conclusion of Theorem \ref{MainTheorem}.

The paper is organized as follows. The preliminary Section 2 begins with some (surprisingly
relevant) generalities on finite dimensional algebras. Then
it  introduces important notions concerning stability for modules attached
to a normal $k$-subgroup scheme of an algebraic group. Most of these concepts can be (or already have been) formulated for abstract groups.
 However, we present them in such a way that they easily extend to $k$-group schemes (and algebraic groups), guided in part by the
functorial viewpoint of \cite{DG}. (This paper accordingly works in an (almost)  ``distribution algebra free" environment.)  The main results are proved in Section 3. Theorem \ref{MainTheorem} itself is a
consequence of Theorem \ref{realmaintheorem}. The latter result, which is established for an affine algebraic group $G$ and a normal $k$-subgroup scheme $N$ such that quotient $G/N$ is reductive, provides an equivalence of three different concepts of stability, one of which
is $G$-stability in a strong form. The
conclusion of Theorem 1.1 appears in the form of ``numerical stability." A third notion, which we call ``tensor stability,"
plays a key role in the proof of Theorem \ref{realmaintheorem}; essentially, we remove cohomological obstructions using
tensor products. These arguments are new. Theorem \ref{finalthm} and Remark \ref{finalthmremark} revisit
the main result of \cite{Donkin1} on the characters of $G_r$-projective covers $Q$ of the irreducible modules.
By using Lemma \ref{firstlemma}, a filtration of $Q$ appearing implicitly in \cite{Donkin1} can be identified
 as the $G_r$-socle filtration. As a result of this identification, it follows that the associated graded module $\gr^{G_r} Q$ has
 a compatible rational $G$-module structure---a fact that seems to have gone unnoticed
 \footnote{Moreover, $\gr^{G_r}Q$ is a (non-trivial) graded module for a positively graded version $\gr_{G_r}{\text{\rm Dist}}(G)$ of the
 distribution algebra of $G$; see footnote 11.  Similar remarks apply to the the space $\gr_{G_r}Q$ obtained from the $G_r$-radical
 filtration of $Q$.   The authors
   do not know if the $G_r$-radical and socle series of $Q$ are the same. See \cite[D14]{Jan} for some positive results for $r=1$ and $p>h$, in the presence
   of a version of the Lusztig conjecture.}  Also, we recast the argument for more general normal $k$-subgroup schemes $N$ and rational $N$-modules $Q$ not necessarily PIMs. (There is no requirement here that $G$ or $G/N$ be reductive.)

    In Theorem \ref{endofsec3} and Corollary \ref{cortoendofsec3}, we take up the question of just how far our strong stability
   notion is from an actual $G$-module structure. In fact, we present an explicit non-abelian cohomological obstruction to the
   existence of a compatible  $G$-module structure. That is, the obstruction is trivial precisely when the $G$-module structure exists. This theory provides a direction for a deeper investigation, potentially giving insights to a proof
   or a counterexample to the conjecture that $G_r$-projective
   covers of irreducible modules are $G$-modules. As the proof of Theorem \ref{realmaintheorem} shows, the
         non-abelian obstruction becomes abelian
         in natural inductive settings.

   Finally, Section 4 extends some of the results of this paper to non-connected groups, and collects together some examples and further remarks.

This paper is heavily influenced by Donkin's 1982 paper \cite{Donkin1}. In particular, Donkin originated  what is now the
main argument in our Lemma \ref{biglemma} in his construction of a group he called $G^*$. This group, in our context, is a homomorphic
image of our group $G^\diamond$ which is introduced in \S2. See Theorem \ref{endofsec3} for
a more explicit description. In reference to $G^*$, Donkin stated in his introduction,
``This group is very similar to that considered by Cline in \S1 of [the 1972 paper \cite{Cline}]. Indeed, it is clear
such a construction is possible for a much wider class of algebraic groups and representations of (not necessarily
reduced) normal subgroups than those considered here." As in \cite{Donkin1}, the main import of this paper remains in
the representation theory of reductive groups $G$. However, the principal results
require at most that $G/N$ (and not $G$) be reductive, and all definitions
and supplementary results have, indeed, been placed in the broader context suggested by Donkin's remarks.
While Ed Cline's 1972 paper was written before the start of
our long collaboration with him, we have welcomed this opportunity to pick up a thread of his earlier research.

Finally, the authors heartily thank the referee for many relevant comments which significantly improved
the final version of this paper. In particular, he/she found a substantive gap in the original proof of the main result,
Theorem 3.3, which the authors have now repaired.

\medskip
\begin{center} {\bf{Notation}}\end{center}
\medskip

 Throughout this paper, $k$ is an algebraically closed field of characteristic
$p$. We allow $p=0$.  By an algebraic group, we mean a reduced, algebraic $k$-group scheme. Reductive algebraic groups will always mean connected and reductive (i.~e., a connected algebraic group over $k$ with trivial unipotent radical $R_u(G)$). We say that a reductive algebraic
group $G$ is simply connected if its derived subgroup $G'$ is either trivial (i.~e., $G$ is a torus) or is
semisimple and simply connected. The category of algebraic varieties (reduced algebraic $k$-schemes) is, of course,
fully embedded into the category of $k$-schemes. We prefer to use the terminology ``morphism of $k$-schemes" even when it is evident
that the objects involved are algebraic varieties.  In turn, the category of $k$-schemes is fully embedded into the category of
$k$-functors \cite{DG}. In particular, it is often possible to check properties of $k$-scheme morphisms by examining the associated
$k$-functor maps. This sophisticated point of view actually simplifies many considerations involving algebraic groups and
 their closed $k$-subgroup schemes.

 Given a module $Q$ for an algebra $A$ (or group), we can form the socle series $0=\soc^A_0Q\subseteq\soc^A_{-1}Q
\subseteq\soc_{-2}^AQ\subseteq\cdots$ and the radical series $A=\rad_A^0Q\supseteq\rad^1_AQ\supseteq\rad_A^2Q\supseteq\cdots$. For convenience,
$\soc_{-1}^AQ$ and $\rad^1_AQ$ can be simply denoted $\soc^AQ$ and $\rad_AQ$, respectively.

\section{Preliminaries} This section collects together some preliminary material which will be needed.

 \subsection{Endomorphism algebras}
Let $Q$ be a finite dimensional (left) module for a finite dimensional algebra $A$ over $k$, and form the endomorphism algebra
 $B=\End_A(Q)$. Let $J$ (resp., $J'$) be the annihilator ideal in $B$ of $\soc^AQ$ (resp., $Q/\rad_A Q$).

\begin{lem}\label{firstlemma} With the above notation,
 $J\soc^A_{-n}Q\subseteq \soc^A_{-n+1}Q$ and $J'\rad_A^nQ\subseteq\rad_A^{n+1}Q,$ for all $n\geq 0$.
In particular, both $J$ and $J'$ are   nilpotent. \end{lem}
\begin{proof} We first prove that $J'\rad_A^nQ\subseteq\rad_A^{n+1}Q$. Let $z\in J'$, so that $zQ\subseteq \rad_A Q$. Then
$z\rad_A^nQ=z(\rad A)^nQ=(\rad A)^nzQ\subseteq(\rad A)^n(\rad A)Q=\rad^{n+1}_AQ$.

Next, let $y\in J$, then $y(\soc_{-n}^n)$ has socle length at most
$n-1$, since it is a homomorphic image of $\soc^A_{-n}Q/\soc^A_{-1}Q$. Hence, $y\soc^A_{-n}Q$ is
 contained $\soc^A_{-n+1}Q$. 
 
 The final assertion regarding the nilpotence of $J$ and $J'$ is obvious. \end{proof}

\begin{cor}\label{firstcor}If $\head_AQ:=Q/\rad_AQ$ is irreducible, then $J\subseteq J'=\rad B$. If $\soc^AQ$ is irreducible, then $J'\subseteq J=\rad B$.
In particular, if both $\head_A Q$, $\soc^AQ$ are irreducible, then $J=J'=\rad B$ and, for all $n\geq 0$,
$$\begin{cases}\rad B\rad^n_AQ\subseteq\rad^{n+1}_AQ, \,{\text{\rm and}}\\
\rad B\soc^A_{-n}Q\subseteq\soc^A_{-n+1}Q.\end{cases}$$
\end{cor}

 \subsection{Stability}
 For an abstract group $G$ with normal subgroup $N$, let $Q$ be a $kN$-module defined by
 a homomorphism  $\rho=\rho_N=\rho_{N,Q}:N\to GL_k(Q)$. For
$g\in G$, let $\rho^g:N\to GL(Q)$ be the representation defined by $\rho^g(n)=\rho(gng^{-1})$. Then $Q$ is called $G$-stable if
$\rho$ is equivalent to $\rho^g$, for all $g\in G$. Explicitly, this means there is a mapping $\alpha:G\to GL_k(Q)$
such that
\begin{equation}\label{stable}\begin{cases} (1)\,\,
\alpha(g)\rho(n)=\rho({^gn})\alpha(g),\quad\forall g\in G, n\in N\\
(2) \,\,\alpha(1)=1=1_Q.\end{cases}\end{equation}
(Here $^gn:=gng^{-1}$.)
We often write $Q^g$ for the $kN$-module defined by $\rho^g$. The underlying vector space of $Q^g$ 
is $Q$. Equation (\ref{stable})(1) says that, for each $g\in G$,
$\alpha(g)$ is an $N$-module isomorphism from $Q$ to $Q^g$; we also write $\alpha(g):\rho\overset\sim\to\rho^g$.
 Replacing $\alpha$ by $\alpha(1)^{-1}\alpha$, we can (and always
will) assume that $\alpha(1)=1$, which is condition (\ref{stable})(2).   If $Q$ is a $G$-module (by an action extending that of $N$), then $Q|_N$ is obviously $G$-stable. 

In the sequel, we call $\alpha:G\to GL_k(Q)$ a {\it structure map} for the $G$-stable module $Q$. In general, it is
{\it not} uniquely determined by the conditions (\ref{stable}). In various situations discussed below, further conditions will be required of a structure 
map $\alpha$. 

\subsubsection{Rational stability} We recast the elementary notion of stability for abstract groups into
the context of $k$-group schemes, or even $k$-groups.

      Thus, suppose $N$ is a closed, normal $k$-subgroup scheme of a $k$-group scheme $G$, and let $Q$ is a finite dimensional rational $N$-module, i.~e., there is a homomorphism $\rho:N\to GL_k(Q)$ of $k$-groups. For any commutative $k$-algebra $S$, the homomorphism $\rho$ induces (by restriction to the category of commutative $S$-algebras) a homomorphism $\rho_S=\rho_{N,Q,S}:N_S\to GL_S(Q\otimes S)$ of $S$-groups. (Here $N_S(T)=N(T)$, for any commutative $S$-algebra $T$.) Thus, $\rho_S$ defines an $N_S$-representation.
For any $g\in G(S)$, there is also an $N_S$-representation $\rho^g_S$ defined as above.
We say that $\rho_S$ is $G(S)$-stable provided
$\rho_S$ is equivalent to $\rho_S^g$, for all $g\in G(S)$. In addition, we say $Q$ is rationally $G$-stable provided
equivalences $\alpha_S(g):\rho_S\overset\sim\to \rho_S^g$ can be chosen to be functorial in $S$.
In other words, $\alpha_S$ is the value at $S$ of a natural transformation from the $k$-functor $G$ to the $k$-functor $GL_k(Q)$---that
 is, $\alpha$ is a $k$-functor morphism---with an additional property. This additional property may be described as the equality of two
 functor morphisms $G\times N\begin{smallmatrix} \to\\ \to\end{smallmatrix} GL_k(Q)$, where the top (resp., bottom)
morphism sends an $S$-point $(g,n)$ to $\alpha_S(g)\rho_S(n)$ (resp., $\rho_S(gng^{-1})\alpha_S(g)$).  The
same property may be expressed diagrammatically for $k$-schemes without the use of functors, starting from a $k$-scheme
morphism $\alpha:G\to GL_k(Q)$. As in the abstract group case, we can (and always will) assume the additional condition that $\alpha_S(1)=1_{Q}$, for all commutative $k$-algebras $S$. (This, too, can
be expressed diagrammatically.)

In practice, we will blur the distinction between the $k$-scheme $G$ and the associated $k$-functor, trusting the reader can
determine the intent from context.  If $Q$ is rationally $G$-stable, through a morphism $\alpha:G\to GL_k(Q)$ of $k$-schemes as above, we say that $\alpha$ affords (or gives) the rational $G$-stability of $Q$. Alternatively, we
say $\alpha$  is a structure map demonstrating the rational $G$-stability of $Q$. Rational $G$-stability of $Q$ is a subtle issue.

Moreover, there is an additional subtlety. In the abstract group case, it is easy also to arrange\footnote{For example, let $\overline G$ be a set of coset representatives for $N$ in $G$ with $1\in\overline G$. Given $g\in G$, let  $\overline g\in \overline G$
belong to the coset to which $g$ belongs. Given $\alpha:G\to GL_k(Q)$, define $\alpha':G\to GL_k(Q)$ by $\alpha'(\overline{g}m)=
\alpha(\overline g)\rho(m)$, for $\overline g\in \overline G$ and $m\in N$. Then, given $g\in G, n\in N$, write $g=\overline g m$ with $m\in N$, so that $\alpha'(gn)=\alpha'(\overline g)\rho(mn)=\alpha'(\overline g)\rho(m)\rho(n)=\alpha'(\overline g m)\rho(n)=\alpha'(g)\rho(n)$. 

One checks that $\alpha'(g):Q\to Q^g$ is an equivalence of $N$-modules (equivalently, satisfies the
condition (\ref{stable})(1)), using the corresponding property for $\alpha$; also, $\alpha'(1)=1$ .} that the following two
equivalent\footnote{The conditions are equivalent in the presence of (\ref{stable})(1). Also, both conditions
together imply (\ref{stable})(1).} conditions hold:
\begin{equation}\label{compatible}\begin{cases} \alpha(gn)=\alpha(g)\rho(n),\\
\alpha(ng)=\rho(n)\alpha(g),\end{cases}\quad\forall n\in N, g\in G.\end{equation}
 This fact implies, in particular, that $\alpha|_N=\rho$, since $\alpha(1)=1$. We sometimes refer
 to second equation above as ``the left-hand version" of (\ref{compatible}). 

 Returning to the  $k$-group or $k$-group scheme case, we can at least imitate the abstract
group arrangement by requiring, for some $\alpha$ affording a
rational $G$-stability, that the same equation  $\alpha_S(gn)=\alpha_S(g)\rho_S(n)$ (equivalently, $\alpha_S(\rho_S(n)g)=\rho_S(n)\alpha_S(g)$) holds, for all $S$-points $g\in G(S),n\in N(S)$, and
for all commutative $k$-algebras $S$. 

Another equivalent requirement is that there is a commutative diagram
   \begin{equation}\label{strong}\begin{CD}  G\times N @>^{\text{\rm mult}}>> G \\
@V{\alpha\times 1_N}VV  @VVV^\alpha\\
  GL_k(Q)\times N @>>_{{\text{\rm mult}}\circ(1\times\rho)}>  GL_k(Q)\end{CD}\end{equation}
  of $k$-schemes (or $k$-functors).
In this case,  $Q$ is called {\it strongly $G$-stable}. (In particular, strong $G$-stability includes,  by definition, rational
$G$-stability.)


  Next, suppose that $Q$ is rationally $G$-stable, and consider $G$-submodules $V$ of $Q$, i.~e., rational $G$-modules $V$ such that $V|_N\subseteq Q$.   The pair $(Q,V)$ is called rationally $G$-stable if a morphism $\alpha:G\to GL_k(Q)$ affording the rational $G$-stability
  of $Q$ can be chosen so that the diagram (with the evident maps)
  \begin{equation}\label{ratstab}
  \begin{array}{ccccc}

G\times Q &

\stackrel{}{\begin{picture}(20,0)\put(0,4){\vector(1,0){20}}\end{picture}}

& GL_k(Q)\times Q &

\begin{picture}(20,0)\put(0,4){\vector(1,0){20}}\end{picture}

& Q \\[1mm]

\begin{picture}(0,20)\put(0,0){\vector(0,1){20}}\end{picture}

&&&&

\begin{picture}(0,20)\put(0,0){\vector(0,1){20}}\end{picture}

\\[1mm]

G\times V &

\multicolumn{3}{c}{\begin{picture}(65,0)\put(0,4){\vector(1,0){65}}\end{picture}}

& V

\end{array}
  \end{equation}
  commutes.    If, in addition, $\alpha$ can be chosen demonstrating that $Q$ is strongly $G$-stable (i.~e., the diagram
   (\ref{strong}) is commutative),
the pair $(Q,V)$ is called {\it strongly $G$-stable.} Strongly $G$-stable pairs play an especially important
role in this paper. For the convenience of the reader, we record, in terms of $S$-points, all the conditions
for a $k$-scheme morphism $\alpha:G\to GL_k(Q)$  to be a structure map demonstrating strong
$G$-stability for the pair $(Q,V)$:
\begin{equation}\label{Spoint}
\begin{cases} (1)\quad \alpha_S(g)\rho_S(n)=\rho_S({^gn})\alpha_S(g), \alpha_S(1)=1_{Q_S}\quad{\text{(rational $G$-stability of}}\,\,Q);\\
(2)\quad \alpha_S(gn)=\alpha_S(g)\rho_S(n)\quad{\text{\rm (diagram (2.2.3))}}; \\
(3) \quad \alpha_S(g)v=gv \quad{\text{\rm (diagram (2.2.4)),}}\\
\end{cases}
\end{equation}
for all $g\in G(S), n\in N(S), v\in V(S)$, and commutative $k$-algebras $S$. Here $1_{Q_S}$ denotes the
identity in $GL_S(Q_S)$.

In Section 3, all the strongly $G$-stable pairs $(Q,V)$ encountered will satisfy the additional condition that
$\soc^NQ\subseteq V$. This implies that $\soc^NQ=\soc^NV$ has a rational $G$-module structure, which is critical
for the argument of Lemma \ref{biglemma}.  

\subsubsection{Numerical stability} Return to the abstract setting of a group $G$, normal subgroup $N$, and a finite dimensional $kN$-module $Q$. We say $Q$ is {\it numerically $G$-stable} provided that exists a $kG$-module $M$ such that $M|_N\cong
Q^{\oplus n}$, for some positive integer $n$. In this case, a Krull-Schmidt argument shows that $Q$ is $G$-stable. Conversely, if
$N$ has finite index in $G$, and if $Q$ is $G$-stable, then it is numerically $G$-stable. In fact, $\res^G_N\ind_N^GQ$
is a direct sum of copies of $Q^g\cong Q$, for $g$ ranging over a set of coset representatives of $N$ in $G$.

Now let $V$ be a $G$-submodule of $Q$, i.~e., $V$ is a $kG$-module whose restriction to $N$ is a submodule of $Q$.  Then the pair $(Q,V)$ is called {\it numerically $G$-stable} if $Q$ is numerically $G$-stable and $M$
can  be chosen so that $M=Q\oplus R$, for some $kN$-module $R\cong Q^{\oplus n-1}$, and such that
$V$ is a $G$-submodule of $M$ contained in $Q$.\footnote{ The definition of a numerically $G$-stable pair is motivated by
both by its applications and its validity in the finite group case: Suppose $N$ has finite index in $G$ and $Q$ is $G$-stable. Then the
pair $(Q,V)$ is numerically $G$-stable, for any $G$-submodule $V$ of $Q$.
For the proof, embed $V$ diagonally in $\res^G_N\ind_N^GQ$ and extend this embedding to $Q$ using $\alpha$. The
image of $Q$ together with the original $N$-module $R$ (obtained from $Q^{\oplus n}=\res^G_N\ind^G_NQ$)   works.}


The definitions are easy enough to give for $k$-group schemes. Suppose $G$ is a $k$-group scheme and $N$ is a closed, normal $k$-subgroup scheme.
Let $Q$ denote a finite dimensional rational $N$-module. Then we say that $Q$ is numerically $G$-stable provided there exists a finite dimensional
 rational $G$-module $M$ such that $M|_N\cong Q^{\oplus n}$, for some positive integer $n$. If $V$ is a $G$-submodule of $Q$, the pair $(Q,V)$ is numerically $G$-stable
provided that $Q$ is numerically $G$-stable. (It should always be clear from context that the
 modules involved are intended to be rational.)

\subsubsection{Tensor stability}\label{tensor}  Again we start with the abstract setting. Let $N$ be a normal subgroup of a group $G$. A finite dimensional $kN$-module
$Q$ is called {\it tensor $G$-stable} provided there exists a nonzero finite dimensional $kG/N$-module $Y$ such that $Q\otimes Y$ is a $G$-module whose restriction
to $N$ identifies with $Q\otimes Y|_N$.
 If $V$ is a $G$-submodule of $Q$, then the pair $(Q,V)$ is called {\it tensor $G$-stable} provided
$Q$ is tensor $G$-stable by a nonzero $G/N$-module $Y$ such that $V\otimes Y$ is a $G$-submodule of $Q\otimes Y$.

Tensor stability for $Q$ is equivalent to numerical stability for $Q$: If $Q\otimes Y$ is a $kG$-module (extending the
 action of $N$) for some nonzero finite
dimensional  $kG/N$-module $Y$, then $Q^{\oplus n}$ is a $kG$-module (extending the action of $N$) if $n=\dim Y$. Conversely, if $Q^{\oplus n}\in kG$-mod,
for some positive integer $n$, then $Q\otimes Y\in kG$-mod, taking $Y\in kG/N$-mod to be $n$-dimensional with
trivial $G/N$-action.

A similar equivalence holds for pairs $(Q,V)$ as a consequence of the following lemma. After we discuss the
above definition in the $k$-group scheme setting, we will observe that the lemma holds in that context as well.

\begin{lem}\label{help} Suppose that $(Q,V)$ is tensor $G$-stable, for some $N$-module $Q$ and $G$-submodule $V$. Then the pair $(Q,V)$ is
numerically $G$-stable. Conversely, if a pair $(Q,V)$ is numerically $G$-stable, it is tensor $G$-stable.\end{lem}

\begin{proof} Suppose that $(Q,V)$ is tensor $G$-stable by means of a nonzero finite dimensional $G/N$-module $Y$. Then $(Q,V)$
is also tensor $G$-stable using the nonzero $G/N$-module $Y\otimes Y^*$. But $V\otimes Y\otimes Y^*\cong V\otimes\End_k(Y)$ contains
$V\cong V\otimes 1_V$ as a $G$-submodule, so that, if $n=\dim Y\otimes Y^*$, then $Q^{\oplus n}$ contains an $N$-direct
summand $Q$ containing $V$ as a $G$-submodule.

Conversely, suppose $(Q,V)$ is numerically $G$-stable. With the notation above defining numerical 
stability, let $Y$ the $n$-dimensional $G/N$-trivial module. \end{proof}

There may be multiple $G$-structures on $Q^{\oplus n}$, especially for different choices
of $n$. Indeed, it is common in the Clifford theory of finite group representations \cite{Cline} to consider the
category $\mathcal C$ of all finite dimensional $kG$-modules $M$
with $M|_N\cong Q^{\oplus n}$, for some fixed $N$-module $Q$. Tensor products play an important role, and the category $\mathcal C$ is stable under tensor products
by $G/N$-modules. In particular, it can contain many non-isomorphic $G$-modules (a fact we have exploited in a positive way in
Lemma \ref{help}).

The tensor stability notion and Lemma \ref{help}
have analogous versions if $G$ is a $k$-group scheme and $N$ is a closed, normal $k$-subgroup scheme.
In fact, the definitions as well as the statement and proof of Lemma \ref{help} carry over verbatim, working
in the categories of rational modules.

Once again, it is interesting to consider the category of finite dimensional rational $G$-modules $M$ such that $M|_N
\cong Q^{\oplus n}$, for some positive integer $n$ and fixed rational $N$-module $Q$. We will see that, while the
goal of Theorem \ref{MainTheorem} is ostensibly numerical stability, its proof is easier to approach via tensor stability.

 \subsection{Schreier systems: the case of abstract groups} We begin by considering various situations involving abstract
 groups and then discuss in the next subsection how to formulate the concepts in the context of  $k$-group schemes.

\subsubsection{The basic set-up} Suppose two groups $G$ and $U$ are given. Schreier---see
\cite[\S15.1]{Hall}---gave conditions for defining a group extension of $G$ by $U$. We follow this procedure closely, though
we use left actions instead of right, and use a formulation more transparently extending to the case of algebraic groups.
There are two ingredients: a ``conjugation action" of $G$ on $U$, and a ``factor set" for this action. The conjugation action
may be viewed as a map $\kappa:G\times U\to U$, $(g,u)\mapsto {^gu}$. We require that 
$$^1u=u\,\,{\text{\rm and}}\,\, {^g(uv)}=({^gu})({^gv}),\quad\forall
u,v\in U, g\in G.$$
 The factor set is a map $\gamma:G\times G\to U$. The pair $(\kappa,\gamma)$
is required to satisfy the following additional conditions:\footnote{As noted in \cite{Hall}, condition (3) is largely simplifying and
can be omitted, provided we also require $^1u=\gamma(1,1)u\gamma(1,1)^{-1}$ instead of $^1u=u$, for all $u\in U$.}
\begin{equation}\label{Schreir}\begin{cases}(1) \quad^g({^hu})=\gamma(g,h)({^{gh}u})\gamma(g,h)^{-1};\\
(2)\quad ^f\gamma(g,h)\gamma(f,gh)=\gamma(f,g)\gamma(fg,h);\\
(3) \quad \gamma(1,1)=1.\end{cases}\,\,\forall f,g,h\in G, \,u\in U.\end{equation}
Thus, if $g\in G$ is fixed,
the map $U\to U$, $u\mapsto {^gu}$, is an automorphism of $U$. Also, the above identities imply that
\begin{equation}\label{cocyclevanishing}\gamma(1,g)=\gamma(g,1)=1,\quad\forall g\in G.\end{equation}
In the special case when $U$ is abelian, $U$ becomes a
(multiplicatively written) abelian group module for $G$, and $\gamma$ defines a classical 2-cocycle\footnote{More precisely,
 $\gamma$ is a ``normalized" $2$-cocycle because of the equations $\gamma(1,h)=\gamma(h,1)=1$, for all $h\in G$. In particular, $(u,g)=(u,1)(1,g)$ for all $u\in U, g\in G$. Also, when $U$ is abelian, the condition (\ref{Schreir})(1) becomes 
 simply
$$ \begin{aligned} (1^\prime)\quad {^g(}{^hu})={^{gh}u}.\end{aligned}
 $$} (and thus
an element in $H^2(G,U)$.) In the general (non-abelian) case, we call $(\kappa,\gamma)$ a Schreier system for the pair $(G,U)$.

A Schreier system $(\kappa,\gamma)$ on $(G,U)$ defines a group structure on the set $U\times G$, with multiplication, inverses, and identity $1$ given explicitly by
\begin{equation}\label{groupoperations}\begin{cases}(1)\quad (x,g)(y,h):=(x\,({^{g}y})\gamma(g,h),gh),\quad x,y\in U, g,h\in G;\\
(2)\quad (x,g)^{-1} :=({\gamma(g^{-1},g)^{-1}}\,({^{g^{-1}}x})^{-1},g^{-1}),\forall x\in U,g\in G ;\\
(3)\quad 1=(1,1).\end{cases}\end{equation}
Denote this extension group by $G^\diamond=G^\diamond(\kappa,\gamma,U)$.
In this way, we obtain an exact sequence
\begin{equation}\label{shortexact}1\to U\to G^\diamond\overset\pi\to G\to 1\end{equation}
of groups. The mapping $\iota:G\to G^\diamond$ defined by $\iota(g)=(1,g)$ provides a set-theoretic section of $\pi$, satisfying
the additional condition $\iota(1)=1$.
Conversely, any exact sequence $1\to U\to G^\flat\overset\pi\to G\to 1$ of groups, together with a set-theoretic section $\iota$ of $\pi$ with $\iota(1)=1$, arises from a Schreier system for $(G,U)$. In fact, identifying $U$ with its image in $G^\flat$, we have an identification $G^\diamond\to G^\flat$ given by
$(u,g)\mapsto u\iota(g)$. It is useful to note, in anticipation of the algebraic group case, that any section $\iota:G\to G^\flat$ of $\pi$ can be easily modified to a section $\iota'$ satisfying $\iota'(1)=1$. (Put $\iota'(g)=\iota(1)^{-1}\iota(g)$,
for example.)

Continuing to follow \cite{Hall}, a Schreier system $(\kappa,\gamma)$ for $(G,U)$ is said to be equivalent to a Schreier system
$(\kappa',\gamma')$ for $(G,U)$ if there is a map $\beta:G\to U$ satisfying
$$\begin{cases} \kappa'(g,x)=\beta(g)\kappa(g,x)\beta(g)^{-1};\\
\gamma'(g,h)=\beta(g)\kappa(g,\beta(h))\gamma(g,h)\beta(gh)^{-1},\end{cases}\,\,\forall x\in U, g,h\in G.$$

The Schreier system $(\kappa,\gamma)$ for $(G,U)$ is called split provided there is a group homomorphism $\sigma:G\to G^\diamond$
such that $\pi\circ\sigma=1_G$. Necessarily, $\sigma(g)=(\beta(g),g)$, for some mapping $\beta:G\to U$, and
$\gamma(g,h)=\kappa(g,\beta(h))^{-1}\beta(g)^{-1}\beta(gh)$, for all $g,h\in G$.

\subsubsection{Inflation} Let $N$ be a normal subgroup of $G$, and let the natural quotient homomorphism $G\to G/N$ be denoted by
$g\mapsto\overline g$. We say that a Schreier system $(\kappa,\gamma)$ for $(G,U)$ is the
inflation of a Schreier system $(\kappa',\gamma')$ for $(G/N,U)$ provided that $\kappa(g,u)=\kappa'(\overline g,u)$
and $\gamma(g,h)=\gamma'(\overline g,\overline h)$, for all $g,h\in G$, $u\in U$. Given any Schreier system
$(\kappa',\gamma')$ for  $(G/N,U)$, these formulas define a Schreier system $(\kappa,\gamma)$ for $(G,U)$---called the
inflation of $(\kappa',\gamma')$ to $G/N$.\footnote{There is a more general notion of inflation, which might be called ``factor set inflation." This requires only that
$\gamma:G\times G\to U$ be defined by inflation of $\gamma':G/N\times G/N\to U$. It is still true that $\iota|_N$ is a group
 isomorphism of $N$ onto its image $\iota(N)$. However,  $\pi^{-1}(N)$ is now only
 a semidirect product $U\ltimes \iota(N)$ of $U$ and $\iota(N)$, so that the conjugation of $\iota(N)$ on $U$ may not
 be trivial. We will have no use of this more general version of inflation in this paper.}
 
  Let $G^\diamond$ be the extension group for the induced Schreier system $(\kappa,\gamma)$.
Let $\iota:G\to G^\diamond$ be the set-theoretic section for $\pi$ defined by $\iota(g)=(1,g)$. Observe that
$\iota|_N$ is a group homomorphism, mapping $N$ isomorphically onto its image $\iota(N)$ which commutes elementwise with $U$, using the identity $\kappa(n,u)=\kappa'(1,u)=u$, for all $u\in U,n\in N$.  The preimage $N^\diamond:=\pi^{-1}(N)$ in $G^\diamond$ is thus naturally isomorphic to 
the group direct product $U\times \iota(N)$.
 
 Finally, (\ref{Schreir})(2) implies that $^g\gamma(g^{-1},g)=\gamma(g,g^{-1})$, so that
 $$\begin{aligned} (1,g)(1,n)(1,g)^{-1} & = (1,gn)(\gamma(g^{-1},g)^{-1},g^{-1})\\
 & = ({^g\gamma}(g^{-1},g)^{-1}\gamma(g,g^{-1}),gng^{-1})\\&= (\gamma(g,g^{-1})^{-1}\gamma(g,g^{-1}),gng^{-1})\\
 &=(1,gng^{-1}),
 \end{aligned}
$$
which easily shows, since $\iota(N)$ commutes elementwise with $U$,  that $\iota(N)$ is normal in $G^\diamond$.

\subsubsection{Schreier systems arising from representations }\label{SchreirRep} Let $N$ be a normal subgroup of a group $G$ and suppose that
$Q$ is a $G$-stable $kN$-module with respect to a mapping $\alpha:G\to GL_k(Q)$ satisfying (\ref{stable}), where $\rho:N\to GL_k(Q)$
defines the action of $N$ on $Q$. Thus, for $g\in G$, $\alpha(g):Q\overset\sim\to
Q^g$ is an isomorphism as $N$-modules, and $\alpha(1)=1_Q$. As pointed out earlier (see footnote 6), we can assume $\alpha$ satisfies the condition (\ref{compatible}) or,
equivalently, the diagram (\ref{strong}) is commutative.

Define $\gamma:G\times G\to GL_k(Q)$ by putting
\begin{equation}\label{repfactorset}
\gamma(g,h):=\alpha(g)\alpha(h)\alpha(gh)^{-1}\in GL_k(Q),\,\forall g,h\in G.\end{equation}
 Observe that $\gamma(g,h)$ satisfies $^{\gamma(g,h)}\rho(n)=\rho(n)$,
for all $n\in N$. Thus, $\gamma(g,h)\in \text{Aut}_{kN}(Q)\subseteq GL_k(Q)$.  Also, $\alpha(g)\text{Aut}_{kN}(Q)\alpha(g)^{-1}=
\text{Aut}_{kN}(Q)$. Let $U$ be any subgroup of $\text{Aut}_{kN}(Q)$, such as $U=\text{Aut}_{kN}(Q)$, which is stable under all
conjugations by all elements $\alpha(g)$, and contains all $\gamma(g,h)$. Identify $\gamma$ with the map
$G\times G\to U$ it induces, keeping the same notation. Define $\kappa:G\times U\to U$ by
putting
\begin{equation}\label{kappadefined}\kappa(g,u)=
\alpha(g)u\alpha(g)^{-1},\quad\forall g\in G, u\in U.\end{equation}
 Then $(\kappa,\gamma)$ is a Schreier system for the pair $(G,U)$. In fact,
$(\kappa,\gamma)$ is clearly inflated from a Schreier system for $(G/N,U)$. Let $G^\diamond=G^\diamond(\kappa,\gamma)$ denote the extension group of $G$ by $U$ defined
by $(\kappa,\gamma)$. We also write $G^\diamond=G^\diamond(\alpha,U)$, noting that $\kappa$ and
$\gamma$ are determined by $G,\alpha,U$. Then $G^\diamond$ acts naturally on $Q$ by $k$-automorphisms, viz., 
$$(u,g)q:=u\alpha(g)q,\quad u\in U,g\in G,
q\in Q.$$
 Let $\rho:G^\diamond\to GL_k(Q)$ denote this representation
In addition, $N^\diamond=\pi^{-1}(N)\cong U\times N$, so that $N$ can be naturally regarded as a 
normal (as noted above) subgroup of $G^\diamond$. The
action of $G^\diamond$ on $Q$ restricts to the original action of $N$ on $Q$ (and both representations are denoted by $\rho$).

Regarding $G$ as a subset of $G^\diamond$ (via $g\mapsto (1,g)$), $\alpha(g)=\rho(g)$, for all $g\in G$, the original $G$-stability of the $kN$-module $Q$ may be recovered from the structure of $G^\diamond$-module structure on $Q$. More generally,
we have the following result. Note the present group $G^\diamond$ fits the hypothesis of the proposition with $Q=Q'$, $\rho=\rho'$,
and $\alpha=\alpha'$.

\begin{prop}\label{notneeded} Let $G^\diamond$ be any group constructed from a Schreier system $(\kappa,\gamma)$ for a pair $(G,U)$ which is
the inflation from a Schreier system for $(G/N,U)$.  Regard $G$ as a subset of $G^\diamond$ as
above. Then any $kG^\diamond$-module $Q'$ is naturally a $G$-stable module for $N$,
afforded by the map $\alpha'=\rho'|_G:G\to GL_k(Q')$, where $\rho':G^\diamond\to GL_k(Q')$ defines
the $G^\diamond$-structure on $Q'$. This map $\alpha'$ automatically satisfies the condition (\ref{compatible}) or, equivalently,
the diagram (\ref{strong}) commutes. \end{prop}

\begin{proof}  Since $N$ is naturally a subgroup
of $G^\diamond$ (as described above), $\rho'|_N$ defines $Q'$ as an $N$-module. It is routine to verify that $\alpha'$ satisfies (\ref{stable}) and (\ref{compatible}), replacing $\alpha,\rho$ by $\alpha',\rho'$.
\end{proof}

\subsection{Schreier systems: the case of $k$-group schemes} The above discussion of Schreier systems has been for abstract groups.
It remains to discuss how all this works for $k$-group schemes.
Suppose that $G$ and $U$ are $k$-group schemes, or, more generally, $k$-group functors.
(In \S3, $G$ and $U$ will both be algebraic groups.)

\subsubsection{The basic set-up}
The definition of
a Schreier system in (\ref{Schreir}) is easily imitated with $\kappa:G\times U\to U$ and $\gamma:G\times G\to U$
required to be maps of $k$-functors, and the required  conditions (such as (\ref{Schreir})) interpreted at
the level of $S$-points, for any commutative $k$-algebra $S$. In this case, $G^\diamond$ acquires the structure of
a $k$-group functor with underlying $k$-functor $U\times G$.

Alternately, the $k$-group scheme case requires $\kappa:G\times U\to U$ and $\gamma:G\times G\to U$ to be
morphisms of $k$-schemes, with all conditions interpreted diagrammatically. For instance, condition (1) of
(\ref{Schreir}) requires the equality of two morphisms $G\times G\times U\to U$, while condition (3) requires
the equality of two maps $e\times e\to U$  (where $e$ is the trivial $k$-group scheme), namely,
 $$\begin{CD} e\times e @>>> G\times G \\
@VVV @ VV\gamma V \\
e @>>> U\end{CD}
$$
In this case, $G^\diamond$ becomes a $k$-group scheme, affine when $G$ and $U$ are affine, and an algebraic group
when $G$ and $U$ are algebraic groups. The remaining details are left to the reader.

\subsubsection{Inflation} When $N$ is a closed, normal $k$-subgroup scheme of $G$, the above discussion on inflation carries through  essentially unchanged, with only some attention to the map $\iota:N\to G^\diamond$. For $k$-group schemes, this is the
$k$-group scheme map given by the composite $N\hookrightarrow G\to e\times G\hookrightarrow U\times G=G^\diamond$. Because $\iota$ is split, it may be factored as an isomorphism followed
by the inclusion of a closed $k$-subgroup scheme, the latter called
$\iota(N)$. Alternately,
we can
get a $k$-functor map $N\to G^\diamond$ in the usual way from the abstract group case.  The splitting again implies a factorization for this map, and $\iota(N)(S)$ may
be taken as the
(isomorphic) image, in the sense of sets and abstract groups, of
the map $N(S)\to G^\diamond(S)$, for any commutative $k$-algebra $S$. This point of view makes it very clear that
$\iota(N)$ is normal in $G^\diamond$.

\subsubsection{Schreier systems arising from representations}Continue to let $N$ be a closed, normal $k$-subgroup scheme of $G$, and
let $Q$ be a finite dimensional rational $N$-module. Assume that $Q$ is strongly $G$-stable, so that there is a morphism
$\alpha:G\to GL_k(Q)$  proving a commutative diagram (\ref{strong}) (or, at the level of $S$-points, for
a commutative $k$-algebra $S$, that
condition (\ref{compatible}) holds). Then $\gamma:G\times G\to GL_k(Q)$, defined as in (\ref{repfactorset}), 
is a morphism of $k$-schemes.  Assume $U$ is a closed subgroup of $\Aut_N(Q)$ which is stable under conjugation by all  elements
$\alpha(g)$, $g\in G$, and ``contains all $\gamma(g,h)$, $g,h\in G$."  Then
$\kappa:G\times U\to U$, defined in (\ref{kappadefined}), is also a morphism
of $k$-schemes. The first condition can be phrased either by a requirement on
$S$-points, for all commutative $k$-algebras $S$, or as a condition on the evident morphism $G\times U\to GL_k(Q)\times U\to GL_k(Q)$, that it factor through the inclusion $U\hookrightarrow GL_k(Q)$. Similar formulations may be given for the second
condition. Then the $k$-group scheme $G^\diamond=G^\diamond(\alpha,U)$ is defined by the
same construction as in the abstract case.

The following explicit set-up will be studied in \S3. Let $G$ be an algebraic group over $k$ and let $N$ be a closed, normal subgroup. Suppose we are given a strongly $G$-stable pair $(Q,V)$ where $V$ is a rational $G$-submodule of $Q$
containing $\soc^NQ$. Let $J_V=J_V^Q\subseteq\End_N(Q)$ be the annihilator of $V$. By Lemma \ref{firstlemma}, $J_V$ is a nilpotent
ideal, so that $1_Q+J_V$ is a closed (unipotent) $k$-subgroup scheme of $GL_k(Q)$, which we denote by $U$. Conditions (\ref{compatible}) imply that $\gamma(g,h)\in U$, for all $g,h\in G$, and that $U$ is stable under conjugation by all $\alpha(g)$, $g\in G$. The corresponding extension group $G^\diamond=G^\diamond(\alpha, U)$ will play an important role in what follows. We also write $G^\diamond=G^\diamond(Q,V,\alpha)$, since the group $U$ is determined by the strongly $G$-stable pair $(Q,V)$. (It turns out, the dependence on $\alpha$ can be largely removed; see Theorem \ref{endofsec3}.)  Given a finite dimensional rational $N$-module $Q$ and a rational $G$-submodule $V$ containing $\soc^NQ$, the strong
$G$-stability of $(Q,V)$ may often be verified using Lemma
\ref{biglemma}. Also, strong $G$-stability of $Q$ is a consequence of the existence of a group like $G^\diamond$, even for a different module $Q'$, using the $k$-group scheme analogue of Proposition \ref{notneeded}. (We mention this only for theoretical completeness, and do
not require it later in the paper.)

Finally, suppose $(Q,V)$ is strongly $G$-stable with structure map $\alpha:G\to GL_k(Q)$. Let $Y$
be a finite dimensional rational $G$-module, defined by $\rho=\rho_{G,Y}:G\to GL_k(Y)$. Then $(Q\otimes Y,V\otimes Y)$ is a strongly $G$-stable pair with structure map 
$\widetilde\alpha:=\alpha\otimes \rho_{G,Y}:G\to GL_k(Q\otimes Y)$. (This will most often be used when $Y$ is   rational $G/N$-module regarded as a $G$-module by inflation through $G\to G/N$). In particular, we can form
the group $G^\diamond(Q\otimes Y,V\otimes Y,\widetilde\alpha)$ as above.

\subsection{Steinberg modules and injective modules} In this
section $k$ has positive characteristic $p$, except in Remark \ref{charzero}. Let $\widetilde G$ be a simply connected reductive group over $k$ with
derived group $\widetilde G'$. Let $T'$ be a maximal torus of $\widetilde G'$ contained in a maximal torus $T$ of $\widetilde G$. Pick a set $\Pi$ of simple roots for $\widetilde G'$, and let $\rho$ be the Weyl weight on $T'$, defined as one-half the
sum of the positive roots. Thus, $\rho$ is a dominant weight on $T'$. Since the restriction map $X(T)\to X(T')$ on
character groups is surjective, we can fix a character (or weight) on $T$ whose restriction to $T'$ is the Weyl weight
$\rho$. By abuse of notation, we denote this weight by $\rho$. For a positive integer $n$, let $\St_n=L((p^n-1)\rho)$
be the irreducible (Steinberg) $\widetilde G$-module of highest weight $(p^n-1)\rho$. Let $G'$ be any semisimple
group over $k$ having the same root system as $\widetilde G'$ but not necessarily simply connected. There is an isogeny
$\widetilde G'\to  G'$ having kernel $K$, a finite closed, central $k$-subgroup scheme of $\widetilde G'$ (and $\widetilde G$).  In fact, $K\leq T'$. Let $G:=
\widetilde G/K$, a reductive group with derived group $ G'$. If $p$ is odd, $(p^n-1)\rho$ lies in the root lattice
of $\widetilde G'$ and so $\St_n$ is an irreducible module for $G'$ and hence for $G$. In any event,
$\St_n\otimes \St_n^*$ does have weights in the root lattice of $\widetilde G'$, and so $\St_n\otimes\St_n^*$ is a rational
$G$-module. (If $\widetilde G$ is semisimple, i.~e., if $\widetilde G=\widetilde G'$, then $\St_n\cong \St_n^*$, i.~e., $\St_n$ is self-dual.)

Let $I=\ind_T^{\widetilde G}k$ be the rational $\widetilde G$-module obtained by inducing the trivial module for $T$ to $\widetilde G$. Then $I$ is an
injective object in the category of rational $\widetilde G$-modules. Also, $I$ identifies with $k[\widetilde G]^T$, the fixed points of
$T$ for its right regular action on $k[\widetilde G]$. Since $K\leq T$ is central in $\widetilde G$, $K$ acts trivially on $I$, and $I$ is
a rational (and injective) $ G=\widetilde G/K$-module. Of course, $I$ contains the injective envelope $I(k)$ of the
trivial module. Part (a) of following result is proved in \cite[II.10.13]{Jan}.

\begin{lem}\label{Steinberg} (a) In the category of rational $G$-modules, $I$ is a directed union of $G$-submodules
(or $\overline G$)
$\St_n\otimes\St_n^*\cong\End_k(\St_n)$, i.~e.,
$I=\ind_T^Gk=\underset\longrightarrow\lim\,\St_n\otimes\St_n^*.$

(b)
Thus, given any finite dimensional rational $G$-module $M$ and positive integer $m$, there exists a
rational $G$-module $Y$ such that the natural map $M\to M\otimes \End_k(Y)$, $v\mapsto v\otimes 1$,
is an inclusion (i.~e., $Y\not=0$) and the induced map $H^m(G,M)\to H^m(G,M\otimes\End_k(Y))$ on cohomology
is the zero map.
\end{lem}

\begin{proof} If the Steinberg modules $\St_n$ are $G$-modules, we can  take $Y=\St_n$, for some large $n$, by (a). (This works
because rational cohomology commutes with direct limits.) Otherwise, let $Y=\St_n\otimes\St_n^*$. In the latter case, note
that  the
natural inclusion $M\hookrightarrow M\otimes\End_k(Y)$ factors through the inclusion $M\hookrightarrow M\otimes \End_k(\St_n)$,
since $1_{\St_n}\otimes 1_{\St_n^*}=1_{\St_n\otimes \St_n^*}$.
\end{proof}

\begin{rem}\label{charzero} Lemma \ref{Steinberg}(b) holds in characteristic 0, taking $Y$ to be the one-dimensional
trivial module.\end{rem}

\section{Main results}
  The following important result gives a way to establish strong
$G$-stability, especially for injective $N$-modules and suitably
characteristic submodules. This lemma is inspired by the work of Donkin \cite{Donkin1}. A critical ingredient,
appearing implicitly in \cite{Donkin1}, is the use of the $G$-structure on $\soc^NQ$ to guarantee injectivity of the
 map $\alpha(g)$ appearing in the proof below. Providing a setting for this argument is the origin of our $(Q,V)$ formalism. Although the construction
 of the morphism $\alpha$ (due to Donkin) below may seem ad hoc, the converse to the lemma holds.  We will sketch
the argument for this converse in Section 4.

\begin{lem}\label{biglemma}Let $G$ be an affine algebraic group with closed, normal $k$-subgroup scheme $N$. Let $Q$ be a finite dimensional, rational $N$-module. Assume that there exists a rational $G$-module $M$ such that $M|_N\cong Q\oplus R$ for some $R$ in $N$-mod, and that there is a $G$-submodule
$V$ of $M$ contained in $Q$ and containing $\soc^NQ$. Then the pair $(Q,V)$ is strongly $G$-stable.  \end{lem}

\begin{proof} Replacing $M$ by the finite dimensional $G$-submodule generated by $Q$,
 it can be assumed that $M$ is finite
dimensional. Define a map $\alpha:G\to \End_k(Q)$ by setting $\alpha(g)(v)=\sigma(g.v)$, for $v\in Q$, where $\sigma:M\to Q$ is the
projection of $M$ onto $Q$ along $R$. We claim that $\alpha(g)\in GL_k(Q)$ for all $g\in G$. In fact, let
$K$ be the kernel of $\alpha(g)$.  If $A$ denotes the enveloping algebra of $N$ in $\End_k(M)$ (the
image of the distribution algebra of $N$), we have $gAg^{-1}m\subseteq Am$, for each $m\in M$. It 
follows that
$$\alpha(g)(AK)=\sigma(gAg^{-1}gK)= \sigma(AgK)=A\sigma(gK)=\alpha(g)K=0.$$
Thus, the kernel $K$ of $\alpha(g)$ is stable under the multiplication by $A$, so it is an $N$-submodule
of $Q$.
Therefore, $K\cap \soc^NQ\not=0$. Choose a nonzero vector $v\in K\cap\soc^NQ$. Then $\alpha(g)(v)=
\sigma(g.v)=g.v\not=0$, a contradiction. Thus, $\alpha(g)$ is invertible, and we view $\alpha$
as a map $G\to GL_k(Q)$.

Let $e_1,\cdots, e_q,
e_{q+1},\cdots, e_n$ be an ordered basis for $M$ so that $e_1,\cdots, e_q$ (resp., $e_{q+1},\cdots, e_n$) is a basis for $Q$ (resp.,
$R$), and let $g\mapsto [x_{ij}(g)]$ be the corresponding matrix representation of $G$. Thus, each $x_{ij}\in k[G]$. The
$q\times q$-submatrix $[x_{ij}(g)]_{1\leq i,j\leq q}$ defines the action of $\alpha(g)$ on $Q$. Hence, $\alpha:G\to GL_k(Q)$
is a morphism of $k$-schemes.  Straightforward calculations at the level of $S$-points, for commutative
$k$-algebras $S$, show that all the equations of (\ref{Spoint}) hold for the morphism $\alpha$. Thus, $(Q,V)$ is strongly $G$-stable.
\end{proof}
\medskip
\begin{center}{\bf{\large
$\bigstar\bigstar\bigstar$}}\end{center}
\medskip

{\it For the rest of this section,  $G$ is a connected affine algebraic group over the algebraically closed field $k$, and $N$ is a closed,
normal $k$-subgroup scheme. It will often be required that the quotient $k$-group scheme  be a reductive
algebraic group.}

\medskip
\begin{center}{\bf{\large $\bigstar\bigstar\bigstar$}}\end{center}\medskip

 In Lemma \ref{specialcase} and Theorem \ref{realmaintheorem}, $G/N$ is a reductive group.
A typical case arises when $G$ is reductive, and $N$ is
 an infinitesimal subgroup $G_r$ (e.~g., an $r$th Frobenius kernel), for some positive integer $r$ and some ${\mathbb F}_p$-structure on $G$.

The following lemma is key to the  main result, Theorem \ref{realmaintheorem}, and it makes essential use of the homological algebra
of the reductive group $G/N$. The lemma fails for unipotent groups, at least in characteristic 0, as does the theorem; see \S4.2.

\begin{lem}\label{specialcase} Let $N$ be a closed, normal $k$-subgroup scheme of an algebraic group $G$, and assume the quotient group $G/N$ is a reductive algebraic group. Let $Q'$ be a finite dimensional rational $N$-module,
 and let $(Q',V)$ be a strongly $G$-stable pair with $V$ a $G$-submodule of $Q'$
containing both $\soc^NQ'$ and $\rad_NQ'$. 

(a) The pair $(Q',V)$ is tensor $G$-stable, i.~e., there exists a nonzero finite dimensional rational $G/N$-module $Y$ such that $Q'\otimes Y$
has a rational $G$-module structure
with the following properties:

\begin{itemize}
\item[(1)] $(Q'\otimes Y)|_N\cong Q'\otimes Y|_N$, where $Y|_N$ is a trivial $N$-module;

\item[(2)] The subspace $V\otimes Y$ of $Q'\otimes Y$ is a $G$-submodule isomorphic to the tensor product
of the $G$-modules $V$ and $Y$, regarding $Y$ as a $G$-module through inflation through $G\to G/N$.
\end{itemize}
\medskip
(b) Suppose there is a finite dimensional rational $N$-module $Q$ containing $Q'$ as an $N$-submodule, with $(Q,V)$ having a strongly $G$-stable
structure compatible with that of $(Q',V)$, in the sense that 
\begin{equation}\label{COM}\alpha(g)q'=\alpha'(g)q',\quad g\in G, q'\in Q',\end{equation}
for some morphisms $\alpha:G\to GL(Q)$ and $\alpha':G\to GL(Q')$ which are structure maps for $(Q,V)$ and $(Q',V)$,
respectively. Assume also that $\soc^NQ\subseteq V$.  Then a $G/N$-module $Y$ may be chosen so that,
for some rational $G$-module structure on $Q'\otimes Y$, the conclusion of part (a) holds, and
in addition $(Q\otimes Y,Q'\otimes Y)$ has the structure of a strongly $G$-stable pair. \end{lem}

\begin{proof} Let $A$ be the enveloping algebra of $N$ in $\End_k(Q')$, i.~e., the image of the distribution
algebra of $N$ in $\End_k(Q')$. 

We will follow \S2.4.3 below. Let $J_V=J_V^{Q'}$ be the annihilator in $\End_{N}(Q')=\End_A(Q')$ of $V$.  Since $\soc^NQ'\subseteq V$, $J_V$ is
a nilpotent ideal in $\End_{N}(Q')$ by Lemma \ref{firstlemma}. If $\sigma\in J_V$, 
$$\rad_N(\sigma(Q'))=(\rad A)(\sigma (Q')) =\sigma(\rad_A Q')=\sigma(\rad_N Q')\subseteq\sigma(V)=0.$$
The last containment follows from the assumption that $\rad_NQ'\subseteq V$. Thus,
 $\sigma(Q')\subseteq\soc^NQ'$. Hence, for $\sigma,\tau\in J_V$, $\sigma\tau(Q')=0$, i.~e., $J_V^2=0$.  Thus, $U=U_V:=1_{Q'}+J_V$ is a commutative (closed) subgroup of
$GL_k(Q')$, isomorphic to the additive
vector group $J_V$. This subgroup commutes elementwise with $A$ (and with the action of $N$
at the level of $S$-points). In addition, $J_V$ has a structure of a rational $G$-module with action
given by $g\cdot \sigma:=\alpha'(g)\sigma\alpha'(g)^{-1}$, $g\in G$, $\sigma\in J_V$, where $\alpha':G\to
GL_k(Q)$ is a structure morphism for the pair $(Q',V)$. 

Tensoring $Q'$ with any nonzero finite dimensional rational $G/N$-module $Y$, we have $U_{V\otimes Y}:= 1_{Q'\otimes Y} + J_{V\otimes Y}
\cong 1_{Q'\otimes Y}+ J_V\otimes \End_k(Y)$. Here we consider the strongly $G$-stable pair $(Q'\otimes Y, V\otimes Y)$.
If $\alpha':G\to GL_k(Q')$ denotes a structure morphism for $(Q',V)$, 
then $\widetilde\alpha':=\alpha'\otimes \rho_{G,Y}$ gives a structure morphism for $(Q'\otimes Y,V\otimes Y)$; see the last paragraph of \S2.4.3 (wtih $\alpha',Q'$ here playing the roles of $\alpha,Q$ there). The factor set associated
to $\widetilde\alpha'$ is given by $\widetilde\gamma'(g,h)=\gamma'(g,h)\otimes 1_Y$. Here
$\gamma':G\times G\to U$, $(g,h)\mapsto\alpha'(g)\alpha'(h)\alpha'(gh)^{-1}$,  is the factor set associated to $\alpha'$.

Since $U=U_V$ and $U_{V\otimes Y}$ are
 both commutative, both $\gamma'$ and $\widetilde\gamma'$ are (normalized) 2-cocycles for $G$ acting on
  $U=1_{Q'}+J_V$ and $U_{V\otimes Y}=1_{Q'\otimes Y}+J_V\otimes\End_k(Y)$. 
Write $j'(g,h):= 1_{Q'}-\gamma'(g,h)$ and $\widetilde j'(g,h):=j'(g,h)\otimes 1_Y=1_{Q'\otimes Y}-\widetilde \gamma'(g,h)$. Then $j'$ and $\widetilde j'$ are rational $G$-module 2-coycles in $Z^2(G,J_V)$ and $Z^2(G,J_V\otimes\End_k(Y))$, respectively.
Since $(Q',V)$ is strongly $G$-stable, both $j'$ and $\widetilde j'$ are inflated from $2$-cocycles of $G/N$---denote them by the same notation
in $Z^2(G/N,J_V)$ and $Z^2(G/N,J_{V\otimes Y})$, respectively.
 In particular, $\widetilde j'$ defines an element $[\widetilde j']\in H^2(G/N,J_V\otimes \End_k(Y))$. Lemma \ref{Steinberg} says $Y$
 can be chosen so that $[\widetilde j']=0$. Translating back into multiplicative notation, there is a morphism $\beta:G/N\to U_{V\otimes Y}$ of $k$-schemes (varieties, here) such that
$$\widetilde\gamma'(g,h)={^g\beta}(\overline h)^{-1}\, {\beta}(\overline g)^{-1}\beta(\overline{gh}), \,\forall g,h\in G$$ 
where $g\mapsto \overline g$ is the quotient homomorphism $G\to G/N$. (See the discussion immediately
above the start of \S2.3.2.) Since $\widetilde\gamma'$ is
normalized, this equation forces $\beta(\bar 1)=1_{Q'\otimes Y}$.  Let $G^\diamond=G^\diamond(Q'\otimes Y,V\otimes Y,\widetilde\alpha')$ be the group
constructed at the end of  \S2.4.3 (wth the role of $Q\otimes Y$ there played by $Q'\otimes Y$ here,
and with $\widetilde\alpha'$ here used in a similar substitution). Recall $\widetilde\alpha'=\alpha'\otimes\rho_{G,Y}$.  The map $G\to G^\diamond$ given by $g\mapsto (\beta(\overline g),g)$
is a homomorphism of algebraic groups. Composition of this homomorphism with the homomorphism $G^\diamond\to GL_k(Q')$
gives an algebraic group homomorphism $\widetilde\alpha^{\prime\#}:G\to GL_k(Q'\otimes Y)$ with
$$\widetilde \alpha^{\prime\#}(g)=\beta(\bar g)\widetilde\alpha^{\prime}(g), \quad g\in G.$$
This makes $Q'\otimes Y$ a rational $G$-module, and property (1) is satisfied since 
$$\beta(\bar n)=\beta(1)=1_{Q'\otimes Y},\quad \forall n\in N.$$
(This calculation should be done at the level of $S$-points.)  
 
 Also, property (2) is satisfied, since $\beta(g)\in U_{V\otimes Y}$ for $g\in G$, and $U_{V\otimes}$ acts trivially on $V\otimes Y$ by
 construction. This completes the proof of (a).
   
   We now prove (b), keeping the notation of the proof of (a). Let $J^Q_V$ be the annihilator of $V$ in
   $\End_N(Q)$, and let $A^Q$ be the enveloping algebra of $N$ in $\End_k(Q)$. Observe that $Q'/V$ is a completely reducible $N$-module, by hypothesis. Thus, 
   $$(\rad A^Q)(Q'/V)\subseteq(\rad A^Q)\soc^N(Q/V)=\rad_N\soc^N(Q/V)=0.$$
   That is, $(\rad A^Q)Q'\subseteq V$. So $(\rad A^Q)J^Q_VQ'= J^Q_V(\rad A^Q)Q'\subseteq J^Q_VV=0,$
   which gives $J^Q_VQ'\subseteq\soc^NQ\subseteq V\subseteq Q'$. Consequently, restriction
   to $Q'$ defines a $k$-linear and multiplicative map $J^Q_V\to\End_N(Q')$. Its image, which we denote
   as $I_V$, is contained in $J_V$, since $J^Q_VV=0$. The $G$-module structure on $J_V$ is given by
   conjugation by $\alpha'(g)$, $g\in G$. Since $J^Q_V$ is stable under conjugation by $\alpha(g)$, and
   $\alpha(g)|_{Q'}=\alpha'(g)$, by hypothesis, the subspace $I_V$ is a $G$-submodule of $J_V$. (Indeed, the surjection
   $J^Q_V\twoheadrightarrow I_V$ is $\alpha(g)$-equivariant, for each $g\in G$.)
   
The $2$-cocycle $\gamma':G\times G\to U_V=1_{Q'}+J_V$ has the form $\gamma'(g,h)=\alpha'(g)\alpha'(h)\alpha'(gh)^{-1}$. Since $\alpha(g)|_{Q'}=\alpha'(g)$, $\gamma'$  takes values in $1_{Q'}+I_V\subseteq 1_{Q'}+J_V$. Hence, the additive 2-cocycle $j'$ in the proof of part (a) is the image under inclusion
of a $2$-cocycle $j^{\prime\prime}:G\times G\to I_V$. Similarly, $\widetilde j'=j'\otimes1_Y$ is the image of
$\widetilde j^{\prime\prime}=j^{\prime\prime}\otimes 1_Y$. Now in the proof of part (a), choose the
  rational $G/N$-module $Y$ so that the cohomology class $[\widetilde j^{\prime\prime}]\in H^2(G/N,
I_V\otimes\End_k(Y))$ is zero, using Lemma \ref{Steinberg}.  (This vanishing implies that
$[\widetilde j']$ is zero.) Thus, we can take the map $\beta:G/N\to
U_{V\otimes Y}=1_{Q'\otimes Y}+J_{V\otimes Y}=1_{Q'\otimes Y}+J_V\otimes\End_k(Y)$, in the proof of (a), so that it takes values in $1_{Q'\otimes Y}+ I_{V}\otimes\End_k(Y)$.  Put $I_V\otimes \End_k(Y)=I_{V\otimes Y}$. Let $J^{Q\otimes Y}_{V\otimes Y}$ be
the annihilator of $V\otimes Y$ in $\End_N(Q\otimes Y)$; thus, $J^{Q\otimes Y}_{V\otimes Y}=J_V^Q\otimes
\End_k(Y)$. The $k$-group scheme map $1_{Q\otimes Y}+J^{Q\otimes Y}_{V\otimes Y}\to 1_{Q'\otimes Y}+ I_{V\otimes Y}$ identifies at the $k$-scheme level with the surjective linear map of vector spaces
 $J^{Q}_{V}\otimes\End_k(Y)\twoheadrightarrow I_V\otimes\End_k(Y)$,
which is split in the category of vector spaces. Hence, $\beta$ lifts to a morphism $\beta^{Q\otimes Y}:G/N\to 1_{Q\otimes Y}+J^{Q\otimes Y}_{V\otimes Y}$. 

Now, define 
$$\widetilde\alpha^\#:G\to GL_k(Q\otimes Y),\quad g
\mapsto \beta^{Q\otimes Y}(\bar g)\widetilde\alpha(g), \,g\in G,$$
where $\widetilde\alpha(g)= \alpha(g)\otimes \rho_{G,Y}(g)$.   Then $\widetilde\alpha^\#$ is a morphism of
$k$-schemes (varieties, in this case), which serves (like $\widetilde\alpha$) as a structure morphism
for the strongly $G$-stable structure on the rational $N$-module $Q\otimes Y$. (Note that
$\beta^{Q\otimes Y}(\bar g)\in 1_{Q\otimes Y}+J^{Q\otimes Y}_{V\otimes Y}\subseteq \Aut_N(Q\otimes Y)$, and that the map $g\mapsto \beta^{Q\otimes Y}(\bar g)$ factors
through $G\to G/N$.) 

By construction, $\beta^{Q\otimes Y}(\bar g)|_{Q'\otimes Y}=\beta(\bar g)$, $g\in G$, so
$$\widetilde\alpha^\#(g)|_{Q'\otimes Y}=\beta(\bar g)\widetilde\alpha'(g)=\widetilde\alpha^{\prime\#}(g).$$
The right-hand term is $\rho_{G,Q'\otimes Y}(g)$, for the algebraic group homomorphism
$G\to GL_k(Q'\otimes Y)$ giving the rational $G$-module structure on $Q'\otimes Y$. See the proof of
part (a), modified as discussed above, using $I_V$. It now follows that $\widetilde\alpha^\#:G\to GL_k(Q\otimes Y)$ gives
$(Q\otimes Y,Q'\otimes Y)$ the structure of a strongly $G$-stable pair. This proves (b), and 
completes the proof of the lemma.
\end{proof}

The following result is the main result of this paper. As we will see, Theorem \ref{MainTheorem} is a corollary.

\begin{thm}\label{realmaintheorem} Assume $G/N$ is a reductive group.  Let $V$ be a $G$-submodule of
a finite dimensional rational $N$-module $Q$ containing $\soc^NQ$. The following statements about the pair $(Q,V)$
are equivalent:

(1) $(Q,V)$ is strongly $G$-stable;

(2) $(Q,V)$ is tensor $G$-stable;

(3) $(Q,V)$ is numerically $G$-stable.
\end{thm}

\begin{proof} Lemma \ref{biglemma} implies, as a special case, that if $(Q,V)$ is numerically $G$-stable, then it is strongly $G$-stable. Thus,
(3) $\implies$ (1). On the other hand, (2) $\implies$ (3) follows from Lemma \ref{help}.

So it is enough to show that (1) $\implies$ (2).
Assume that $(Q,V)$ is strongly $G$-stable with structure map $\alpha:G\to GL_k(Q)$. The theorem's
hypothesis also gives that
that $\soc^NQ=\soc^N_{-1}Q\subseteq V$.  Let $n$ be the smallest positive integer such that $Q=\soc^N_{-n}Q$.
Let $i\geq 1$ be the largest integer $\leq n$ such that $V\supseteq \soc^N_{-i}Q$. If $i=n$, then $V=Q$, and there
is nothing to prove. Otherwise, put $Q'=V+\soc^N_{-i-1}Q$.   For $g\in G$, the isomorphism
$\alpha(g):\rho_{N,Q}\to\rho_{N,Q}^g$ also gives an isomorphism of modules for the distribution
algebra of $N$. It follows that  $\alpha(g)A\alpha(g)^{-1}=A$, for all $g\in G$, if $A$ is the enveloping algebra 
of $N$ for its action on $Q$. Thus, for $g\in G$,  $\alpha(g)(\soc^N_{-j}Q)=\soc^N_{-j}Q$ for all $j\geq 1$. 
(Notice $\soc^NQ$ is the subspace of $Q$ killed by $\rad A$, etc.) For
$v\in V$, $g\in G$, $\alpha(g)v=g.v$, so $\alpha(g)|_{Q'}$ has image in $GL_k(Q')$ and so defines 
a structure map $\alpha':G\to GL_k(Q')$, making $(Q',V)$ into a strongly $G$-stable pair. By construction,
the compatibility condition (\ref{COM}) in Lemma \ref{specialcase} holds. Also, $\rad_NQ'\subseteq V$,
since $Q'/V$ is a completely reducible $N$-module. 

With this preparation, we will now prove the implication (1) $\implies$ (2) by downward induction on
the socle length of $Q/V$.

By Lemma \ref{specialcase}(b) (whose conclusions include those of part (a)), there is a rational $G/N$-module
$Y$ and a rational $G$-module structure on $Q'\otimes Y$, satisfying

\begin{itemize}
\item[(1)] $(Q\otimes Y)|_N\cong Q\otimes Y|_N$, with $Y|_N$ is a trivial $N$-module;
\item[(2)] The subspace $V\otimes Y$ of $Q'\otimes Y$ is a rational $G$-module isomorphic to the tensor product of the $G$-modules $V$ and $Y$, regarding $Y$ as a $G$-module through $G\to G/N$.
\item[(3)] $(Q\otimes Y,Q'\otimes Y)$ has the structure of a strongly $G$-stable pair.
\end{itemize}
Observe the $N$-socle length of $(Q\otimes Y)/(Q'\otimes Y)$ is the same as for $Q/Q'$, which
is less than that of $Q/V$. By induction there exists a rational, finite dimensional $G/N$-module $Y'$
such that $Q\otimes Y\otimes Y'$ is a rational $G$-module containing $V':=Q'\otimes Y\otimes Y'$ as a $G$-submodule. Finally,
$V'\otimes Y'$ contains $V\otimes Y\otimes Y'$ as a $G$-submodule.  The tensor $G$-stability of $(Q,V)$
follows, and the theorem is proved. \end{proof}

\medskip
\noindent
{\underline{Proof of Theorem \ref{MainTheorem}:} Under the hypothesis of the theorem, Lemma \ref{biglemma} implies that
the pair $(Q,\soc^{G_r}Q)$ is strongly $G$-stable. Then Theorem \ref{realmaintheorem} implies that $(Q,\soc^{G_r}Q)$
is numerically $G$-stable, as required. \qed

\medskip
The following result\footnote{Very recently, the authors have realized that Theorem \ref{finalthm} can be be
significantly strengthened. Not only are $\gr^NQ$ (see the statement of the theorem) and $\gr_NQ$ (see Remark \ref{finalthmremark}(a)) rational $G$-modules, but both are  graded modules for a positively
graded algebra $\gr_NB$, for a suitably large finite dimensional homomorphic
image $B$ of the distribution algebra  $E={\text{\rm Dist}}(G)$ of $G$. As part of the ``sufficiently large"
condition, it is required that  $Q$ be an $\fa$-module, for the homomorphic image $\fa$ of $\text{\rm Dist}(N)$
in $B$, and that the map $\text{\rm Dist}(N)\to GL_k(Q)$ factor through $\text{\rm Dist}(N)\twoheadrightarrow\fa$. 
 This implies that $\gr\fa=\gr_N\fa$ acts naturally on $\gr^NQ$ or $\gr_NQ$. The acton of $\gr_NB$ 
 is then constructed to induce these actions of $\gr\fa$.  (Note this implies $\gr_NQ$ is generated by its term
in grade 0 as a $\gr_N\fa$ or $\gr_NB$-module, so the action of $\gr_NB$ is quite nontrivial.) The action of $E$ on
$\gr^NQ$, provided by the $G$-action of Theorem \ref{finalthm}, is recovered from the
action of $B/\rad_NB=(\gr_NB)_0$ on $\gr^N Q$. A similar statement holds for $\gr_NQ$.  

If $N$ is a finite group scheme, as in the situation of
Theorem 1.1, then the algebra
$$\gr_N{\text{\rm Dist}}(G)=\gr_NE:=\oplus\rad^i{\text{\rm Dist}}(N)E/\rad^{i+1}{\text{\rm Dist}}(N)E$$
makes sense, and it acts on $\gr^NQ$ (or $\gr_NQ$) compatibly with the natural action of
$\gr_N{\text{\rm Dist}}(N)=\gr{\text{\rm Dist}}(N)$, induced by the original action of $N$ on $Q$.   This result establishes a new graded analogue of the Humphreys-Verma 
conjecture mentioned at the beginning of this paper for $G_r$-PIMS.  Details of
the proof will appear elsewhere.   }
} is an easy consequence of our approach.    Although
we state the socle series version of this result, there is a related radical series result; see Remark \ref{finalthmremark}(a).
There is no need  for any reductive
       requirement  on $G/N$  in the rest of this  section.

\begin{thm}\label{finalthm}  Let $(Q,V)$ be a strongly $G$-stable pair such that $\soc^NQ\subseteq V$.
Then $\gr^NQ:=\bigoplus_{n\geq 0}\soc^N_{-n-1}Q/\soc^N_{-n}Q$ has a ``natural" rational G-module structure
compatible with the given action of $N$.
\end{thm}

\begin{proof} In fact, let $G^\diamond=G^\diamond(\alpha,U)$ be the extension of $G$ by $U:=1_Q+J_{V}$. Then $Q$ is a rational $G^\diamond$-module. Since
$N$ identifies as a normal subgroup of $G^\diamond$, $\gr^NQ$ is a rational $G^\diamond$-module. But $U$ acts trivially
 on $\text{gr}^NQ$ by Lemma \ref{firstlemma}, so it is a rational $G\cong G^\diamond/U$-module with a compatible $N$-action.
\end{proof}

\begin{rem}\label{finalthmremark}(a) There is a dual version of the above theorem. Let $G,N,Q$ be as in Theorem \ref{finalthm}, but suppose that $Q$ has an $N$-homomorphic
image $W$ which has the same head as $Q$ and which has a compatible rational $G$-module structure. Now take linear duals. If the
equivalent conditions in Theorem \ref{realmaintheorem} hold for the pair $(Q^*,W^*)$, we conclude that
$\gr^NQ^*$ has a natural rational $G$-module structure compatible with the of $N$. Thus, $\gr_NQ:=\bigoplus_{n\geq 0}\rad^n_NQ/\rad^{n+1}_NQ\cong(\gr^NQ^*)^*$ has a natural $G$-module structure compatible with the given action of $N$.

There is another variation, if we  suppose the hypotheses of Theorem \ref{finalthm} and also that $Q$ has an irreducible
head as an $N$-module. Of course, $\gr_NQ$ still has a rational $G^\diamond$-module structure. But now $J_V\rad^n_NQ\subseteq \rad^{n+1}_NQ$ by Corollary \ref{firstcor}, so that $U$ acts trivially on $\gr_NQ$, and $\gr_NQ$ is a rational $G$-module with a compatible $N$-module structure.
These comments apply, in particular, to the PIMs for $G_r$ (when $G$ is split reductive in positive characteristic).

(b) The word ``natural" is used above in describing the action of $G$ on $\gr^NQ$ or on $\gr_NQ$
because it arises through the action of $G^\diamond$ on $Q$. As we will see in the next theorem, the latter action does not depend on the choice of
$\alpha:G\to GL_k(Q)$ affording strong stability. If $Q$ is a $G$-module, $\alpha$ can be chosen to be
the homomorphism giving the $G$-action. Then the action of $G$ in Theorem \ref{finalthm} agrees
with the action on $\gr^NQ$ induced by the given $G^\diamond$-action. In this way, also, the action of $G$ in the theorem
is natural. More generally, suppose $H$ is a (closed) subgroup of $G$ containing $N$ such that $Q$ is a rational
$H$-module extending the  $N$-module structure of $Q$. In some cases, we can assume that $\alpha|_H$
gives this $H$-module structure. Again, we find the induced action of $H$ agrees with the action obtained in
the above theorem. An important special case occurs when $H=NT$, with $T$ a maximal torus of $G$ and $Q$ an $NT$-module
satisfying the hypotheses of Lemma \ref{biglemma} with $Q$ an $NT$-submodule of $M$. Then $R$ can be arranged to be
an $H$-submodule and $\alpha|_H$ is a group homomorphism. In this case, the action of $H$---and, in particular, $T$---induced
on $\gr^NQ$ agrees with that of Theorem \ref{finalthm}. In this way, we obtain that $Q|_T\cong(\gr^NQ)|_T$, so that
$Q$ has``formal character" equal to that of a rational $G$-module. For example, let $N=G_r$ and $Q$ is $G_rT$-projective cover of an irreducible $G$-module,  this discussion recovers the main result of Donkin \cite{Donkin1}, that the character of $Q$ is that of a rational $G$-module. We have largely repeated his argument. Aside from our ``naturality" discussion (which is avoidable for the character result, by choosing a $T$-stable decomposition in Lemma
\ref{biglemma}), the only
new ingredient\footnote{See, however, footnote 11.} is Lemma \ref{firstlemma}, which shows that the underlying flag stabilized by $U$ can be taken to be the socle (or radical) series of $Q$. \end{rem}

For the two results below, let $(Q,V)$ be a fixed strongly $G$-stable pair afforded by a morphism 
$\alpha:G\to GL_k(Q)$, and with $\soc^NQ\subseteq V$. Let $G^\diamond=G^\diamond(Q,V,\alpha)$ be the extension of $G$ by $U:=1_Q+J_V$ as defined in \S2.4.3. There is a homomorphism
$\rho^\diamond:G^\diamond\to GL_k(Q)$, defined by making $(x,g)\in G^\diamond$ act on $q\in Q$
by $(u,g)q=u\alpha(g)q$, $u\in U, g\in G, q\in Q$.

\begin{thm}\label{endofsec3}  The image $G^*=U\alpha(G)$ of the homomorphism
$\rho^\diamond:G^\diamond\to GL_k(Q)$ is
independent of the choice of the morphism $\alpha$.  Also,  the induced homomorphism $G\to G^*/U$ of
algebraic groups sending
$g\in G$ to the
image $\alpha(g)$ in $G^*/U$ is independent of $\alpha$.

Finally, the algebraic group $G^\diamond$ is itself independent of the
choice of $\alpha$,
up to an isomorphism preserving the action $G^\diamond\times Q\to Q$ of $G^\diamond$ on $Q$. In fact,
there is a pull-back diagram
\begin{equation}\label{pullback}\begin{CD} G^\diamond @>>> G \\
         @VVV   @VVV \\
         G^* @>>> G^*/U
         \end{CD}
         \end{equation}
in the category of affine algebraic groups.  \end{thm}

\begin{proof}Suppose $\alpha':G\to GL_k(Q)$ also gives strong
$G$-stability of $(Q,V)$, so that
$$\alpha(g)\rho(n)\alpha(g)^{-1}=\rho({^gn})=\alpha'(g)\rho(n)\alpha'(g)^{-1},\,\,
\forall g\in G, n\in N.$$
(As usual, equations such as this involving
$k$-schemes can be interpreted
diagrammatically.) It follows that $\alpha(g)^{-1}\alpha'(g)\in\Aut_NQ$, for all $g\in G$. However,
the strong $G$-stability of the pair $(Q,V)$ requires that $\alpha(g)|_V$ and $\alpha'(g)|_V$ both give
the action of $g$ on $V$. Consequently, $\alpha(g)^{-1}\alpha'(g)\in 1_Q+ J_V=U$. This proves that
$G^*=U\alpha(G)$ is independent of the choice of $\alpha$, and the natural induced map $G\to G^*/U$ is also
independent of $\alpha$.

Finally, there is clearly a natural map
$$G^\diamond \to G^*\times_{G^*/U}G,\,\,(x,g)\mapsto (x\alpha(g),g),\,\, x\in U,g\in G.$$
An inverse to this map is given by $(y,g)\mapsto (y\alpha(g)^{-1},g)$, $g\in G, y\in U$. The
first map is both a group homomorphism and a morphism of $k$-schemes. (Note that the group and $k$-scheme structure
of the pull-back $G^*\times_{G/N}G$ is a closed subgroup of the product $G^*\times G$.) The given inverse is at
least a morphism of $k$-schemes. But it is also a group homomorphism, since
it is inverse to a group homomorphism.
\end{proof}

Given a strongly $G$-stable pair $(Q,V)$ with $V\supseteq\soc^NQ$, parts (2) and (3) of the following corollary present two equivalent (possibly non-abelian) cohomological obstructions to the problem
of extending the $G$-module structure on $V$ to all of $Q$. We continue the above notation, with
$G^\diamond=G^\diamond(Q,V,\alpha)$.

\begin{cor}\label{cortoendofsec3}The following statements are equivalent:

(1) The action of $N$ on $Q$ extends to a rational $G$-action,
agreeing with the action of $G$
on $V$.

(2) The exact sequence $1\to U\to G^\diamond\to G\to 1$ of algebraic groups
is split by a homomorphism $G\to G^\diamond$
extending the obvious map $\iota:N\to N^\diamond\subseteq G^\diamond$,
given at the level of $S$-points by
$\iota_S(n)=(1,n),$ $n\in N(S)$, $S$ a commutative $k$-algebra.

(3) The exact sequence
$1\to U\to G^\diamond/\iota(N)\to G/N\to 1$
of algebraic groups is split.
\end{cor}
\begin{proof}It is clear that (2) $\implies$ (1). 

Now assume that (1) holds, and we prove (2).  Let $\rho_G:G\to GL_k(Q)$ be an algebraic group homomorphism
defining the action of $G$ on $Q$ which extends $\rho_N:N\to GL_k(Q)$ and gives the 
original action of $G$ on $V$.  By Theorem \ref{endofsec3}, we can replace
$\alpha$ by $\rho_G$, in defining $G^*$ and the homomorphism $G\to G^*/U$. Using the pull-back diagram (\ref{pullback}), we obtain, from the identity
map $\id_G:G\to G$ and the homomorphism $G\to G^*$, $g\mapsto \rho_G(g)$, a homomorphism $G\to G^\diamond$ which splits the surjection $G^\diamond\to G$.   Restricting to $N\subseteq G$, this
homomorphism has the property that $N\to G^\diamond\to G^*$ is given by $\rho_G|_N=\rho_N$, and $N\to G^\diamond\to G$ is the inclusion.
Since $\iota:N\to G^\diamond$ has the same properties, the constructed map $G\to G^\diamond$ must give $\iota$ on restriction
to $N$. Thus, (2) holds.

Clearly, (2) $\implies$ (3). Suppose that (3) holds, so that there is an algebraic group homomorphism $\sigma:G/N\to G^\diamond/\iota(N)$ with the composition $G/N\to G^\diamond/\iota(N)\twoheadrightarrow G/N$ the identity map.  If $g\in G$, let $\bar g$ denote its image in $G/N$. Then $\sigma(\bar g)=
(\beta(\bar g),\bar g)$ for some morphsim $\beta:G/N\to U$. A straightforward calculation, using the
homomorphism $\sigma$, shows
that 
$$\widehat\sigma:G\to G^\diamond,\quad g\mapsto (\beta(\bar g),g)$$
splits $\pi:G^\diamond\twoheadrightarrow G$ in the category of affine algebraic groups over $k$. The corollary is completely proved.\end{proof}

\begin{rem}\label{uniqueness} Although issues of uniqueness of the $G$-structure on $Q$ (or, on $Q^{\oplus n}$) in cases where there is such a structure are somewhat distant from the focus of this paper, our set-up can
be used to approach this question. If two $G$-structures are given on $Q$, both compatible with its $N$-module
structure, we obtain, under the hypotheses of Theorem \ref{endofsec3}, two complements to $U$ in $G^\diamond$. When
they are conjugate in $G^\diamond$, the two $G$-structures are equivalent. Even when this does not occur, if the hypothesis of
   Theorem \ref{realmaintheorem} is valid, we may follow its proof, but using Lemma \ref{Steinberg} for 
   $m=1$, to establish conjugacy of these complements in
   a larger group $G^\diamond$, associated to $Q\otimes Y$ for a suitable finite dimensional, rational $G/N$-module $Y$.  That is, given two rational $G$-module
   structures $Q'$ and $Q^{\prime\prime}$ with the same underlying $N$-module $Q$, there is, assuming the hypothesis of
   Theorem \ref{realmaintheorem}, an isomorphism $Q'\otimes Y\cong Q^{\prime\prime}\otimes Y$ for some $G/N$-module $Y$.
     (This argument has much in common with that  in Donkin
\cite{Donkin},
which proves uniqueness of PIM $G$-structures---when they exist---after tensoring with a single suitably large twisted Steinberg module. Indeed, the set-up in \cite {Donkin} has the advantage for uniqueness questions of dealing entirely with abelian 1-cohomology issues, as occur with $\Ext^1$-groups.)\end{rem}

\section{Final bits and pieces}

\subsection{The non-connected case} This paper has been written with a view toward applications to connected
groups, but the arguments apply without that assumption. For example, there is no need to require in Lemma \ref{specialcase}
that $G/N$ is connected, only that the connected component $(G/N)_o$ of the identity is a reductive group. Once a $Y$ has
been chosen for $(G/N)_o$, its induction to $G/N$ works as a $Y$ for $G/N$. Similarly, Theorem \ref{realmaintheorem} can
be formulated for non-connected groups $G$ requiring only that $(G/N)_o$ be reductive. With the modified Lemma \ref{specialcase},
the proof of theorem is essentially unchanged. Note that Lemma \ref{biglemma} has no connectedness requirement.

\subsection{Examples} In this subsection, we present several elementary examples.

\subsubsection{$G$-stability does not imply a $G$-module structure}Let $G=G'\times G^{\prime\prime}$, where $G'\cong G^{\prime\prime}\cong SL_2(k)$ and $k$ has characteristic $p=2$. Let $N=G_1=G'_1\times G_1^{\prime\prime}$, the first Frobenius kernel of $G$. The
category of rational $N$-modules identifies with the category of restricted modules for the Lie algebra ${\mathfrak g}'
\times{\mathfrak g}^{\prime\prime}$ of $G$. Let
$Q$ be the space of $2\times 2$ matrices of trace 0 over $k$. Both ${\mathfrak g}'$ and ${\mathfrak g}^{\prime\prime}$ act
on $Q$ via $v\mapsto [X,v]$, since $Q$ identifies naturally with ${\mathfrak g}'$ and with ${\mathfrak g}^{\prime\prime}$.
Furthermore, all these linear operators from ${\mathfrak g}'$ and ${\mathfrak g}^{\prime\prime}$ commute with each other. Hence, $Q$ is a restricted $\mathfrak g$-module, and hence a rational $N$-module. We check directly that $Q$ is $G$-stable, but it does not have the structure of a rational $G$-module compatible with its
$N$-module structure.\footnote{It is easy to see $Q$ is not even strongly $G$-stable. If it were,
   we could construct a group $G^\diamond$ as in section 2, for
   which there would be a group homomorphism into $\text{SL}_k(Q)$, even
   into a maximal parabolic subgroup $P$. The group $G^\diamond$ contains two subgroups $G^{\prime\diamond}$ and $G^{\prime\prime\diamond}$,
pull-backs of $G'$ and $G^{\prime\prime}$ from the natural quotient $G$ of $G^\diamond$. Clearly,
 one of $G^{\prime\diamond}$ or $G^{\prime\prime\diamond}$ maps into $R_u(P)$ in the supposed map of $G^\diamond\to P$. This map has a two dimensional
image on the Lie algebras of each of $G'$ and $G^{\prime\prime}$, each image nontrivial
under the induced adjoint action. But $R_u(P)$ is commutative.}

When $G$ is not reductive it is even easier to give such examples. For example, let $G$ be the group of upper unipotent
$3\times 3$-matrices. Let $N$ be its one-dimensional center. Let $Q$ be the 2-dimensional indecomposable for $N$ in which
$N$ acts as the upper unipotent $2\times 2$-matrices in SL$_2(k)$.  Then $Q$ is $G$-stable, but does not have a $G$-structure compatible with the action of $N$. Indeed, $Q$ is strongly $G$-stable, and the pair $(Q,V)$ is strongly $G$-stable, where $V=\soc^N Q$. However, the pair $(Q,V)$ is not numerically $G$-stable when $k$ has characteristic 0. Using Corollary \ref{cortoendofsec3},
 it comes down to killing a 2-cohomology class with coefficients in $J_V = k$ by embedding in a larger $J=J_Y$, where $Y$ is the fixed point module for $N$ in $M$. It easy to see $J$ has the form $\End_k(Y)$, as a $G/N$-module, where $k$ is embedded as scalar
 multiplications of $1_Y$. But such an embedding of $k$ in such a $J$ is split in characteristic 0, so cannot kill any
nonzero cohomology class.
This shows that Theorem \ref{realmaintheorem} can fail if its
hypothesis that $G/N$ be reductive is removed.\footnote{The example also shows Lemma \ref{specialcase} fails for unipotent groups
  in characteristic 0, since its conclusion implies numerical
  stability. The argument, however, can be viewed even more
  directly in this case, showing clearly why it is not possible
  to kill the underlying cohomological obstruction with a tensor product.}

\subsubsection{Not all indecomposable $N$-modules are $G$-stable} For example, let $G$ be a semisimple, simply connected algebraic
group.
Let $N=G_r$, for some $r\geq 1$. For an $r$-restricted dominant weight $\lambda$, let $Z_r(\lambda)=\text{coind}_{B_r^+}^G\lambda$
in the notation of \cite[II.3]{Jan}. By the discussion in \cite[II.11]{Jan}, $Z_r(\lambda)$ is $G$-stable if and only
if $\lambda=(p^r-1)\rho$, i.~e., $Z_r(\lambda)=\St_r$.

\subsection{Converse to Lemma \ref{biglemma}} We give a brief sketch. Let $G$ be an affine algebraic group (not necessarily connected), and let $Q$ be a finite dimensional rational module for a closed, normal $k$-subgroup scheme $N$ of $G$. Suppose
that $V$ is a $G$-submodule of $Q$ containing $\soc^NQ$. Now assume that the pair $(Q,V)$ is strongly $G$-stable, so there
is a morphism $\alpha:G\to GL_k(Q)$ such that (\ref{Spoint}) holds.    It follows that the left-hand version of (\ref{compatible}) holds,
namely,  $\alpha(ng)=\rho(n)\alpha(g)$,
 for all $g\in G, n\in N$. (This can, and should be, written in terms of $S$-points, but we take
that as understood in this context. We continue with this
  informal mode of exposition below.) Identifying $Q$ with affine $\ell$-space if $\ell=\dim Q$, let $\text{Morph}(G,Q)$ be the
vector space of morphisms $f:G\to Q$ (of $k$-schemes). For $f\in\text{Morph}(G,Q)$ and $n\in N$, define $f\cdot n\in \text{Morph}(G,Q)$
by putting $(f\cdot n)(g)=n^{-1}f(ng)$, for $g\in G$.  In this way, $\text{Morph}(G,Q)$ is a right $N$-module.
The space $\text{Morph}_N(G,Q)$ of $N$-invariant functions is a left $G$-module with respect to the action $(h\cdot f)(g):=
f(gh)$, for $f\in \text{Morph}_N(G,Q)$, $h,g\in G$. Then $\text{Morph}_N(G,Q)\cong\ind_N^GQ$, the rational $G$-module obtained by inducing $Q$ from $N$ to $G$; see \cite{CPS} or \cite[I.3.3(2)]{Jan}.
Further, given $q\in Q$, let $f_q\in \text{Morph}(G,Q)$ be defined by $f_q(g):=\alpha(g)q$, $g\in G$. For $n\in N$,
$(f_q\cdot n)(g)=n^{-1}f_q(ng)=\rho(n^{-1})\alpha(ng)q=f_q(g)$, verifying that $f_q\in \text{Morph}_N(G,Q)$.   (Here we use the left-hand version of (\ref{Spoint})(2).) Furthermore,
the map $Q\to\text{Morph}_N(G,Q)$, $q\mapsto f_q$, is similarly checked, using (\ref{compatible}), to be a morphism of left $N$-modules, which restricts
to a morphism of $G$-modules on the $G$-submodule $V$ of $Q$.
Next, if $\text{Ev}:\text{Morph}_N(G,Q)\to Q$, $f\mapsto f(1)$, is the evaluation map, then the composition $Q\to \text{Morph}_N(G,Q)\overset{\text{Ev}}\to Q$ is the identity map on $Q$. (One needs the fact that $\alpha(1)=1_Q$.) Since Ev is a morphism of $N$-modules, the map $Q\to \text{Morph}_N(G,Q)$ splits as a morphism of $N$-modules. Identify $Q$ with its image in
$\text{Morph}_N(G,Q)$. Let $M$ be the finite dimensional
$G$-submodule of ${\text{Morph}}_N(G,Q)$ generated by $Q$. Then $V\subseteq Q$ is a $G$-submodule of $M$, and we have shown
that $M|_N\cong Q\oplus R$ in the category of rational $N$-modules, i.~e., the converse of the lemma is proved. (Of course, when $G/N$ is
reductive, the converse also follows from Theorem \ref{realmaintheorem}, which establishes numerical
stability---a much
stronger result.)

\subsection{Observations on observability} Let $H$ be a closed subgroup (or, closed $k$-subgroup scheme) of an affine algebraic group $G$ (or,
more generally, a $k$-group scheme $G$). We do not assume that $H$ is normal. Recall that $H$ is observable provided that
every rational $H$-module $V$ is a submodule of a rational $G$-module. There is also a
quotient module condition, namely,  $H$ is observable if and only if, given
any rational $H$-module $Q$, the evaluation map $\text{Ev}:\ind_H^GQ\to Q$ is surjective (which holds
if and only if
$Q$ is a quotient of a rational $G$-module). See \cite{Grosshans} for
more discussion, further references, etc. Though the observable terminology has been used only for subgroups, it applies
in a similar way to rational $H$-modules. We will call a finite dimensional rational $H$-module $Q$ {\it split observable}
provided the evaluation map $\text{Ev}:\ind_H^GQ\to Q$ is surjective and splits as a map of $H$-modules.
(Equivalently, $Q$ is an $H$-direct summand of a rational $G$-module.) It is easy to check $Q$ is
split observable if and only if there exists a finite dimensional rational $G$-module $M$ such that $Q$ is a
direct summand of $M|_H$. 
The proofs of Lemma \ref{biglemma} and its converse above can be easily modified to show that $Q$ is split observable for $H$
if and only if there exists a morphism $\alpha:G\to \End_k(Q)$ of $k$-schemes satisfying $\alpha(1)=1_Q$, $\alpha(gh)=\alpha(g)\rho_H(h)$
and $\alpha(hg)=\rho_H(h)\alpha(g)$, for all $g\in G,h\in H$. (As usual, these equalities must be suitably interpreted for
$k$-group schemes.) In view of footnote 6, these conditions are equivalent to the strong $G$-stability
of $Q$ when $H$ is normal. Thus, split observability provides a generalization of the strong $G$-stability
property to modules  for subgroups $H$ which are not necessarily normal.

\subsection{Schreier systems} Finally, while it has not been our intention to write a treatise on the Schreier construction for general $k$-group schemes, we do note that most of the definitions and constructions of \S2.4 require only that $k$ be a commutative ring.

\end{document}